\documentclass[10pt]{amsart}
\usepackage{amssymb}
\usepackage{amscd}
\usepackage{comment}

\theoremstyle{plain}
\newtheorem{theorem}{Theorem}

\theoremstyle{definition}

\newtheorem*{thm}{Theorem}
\newtheorem{proposition}{Proposition}[section]

\newtheorem{corollary}[proposition]{Corollary}
\newtheorem{lemma}[proposition]{Lemma}
\newtheorem{definition}[proposition]{Definition}
\theoremstyle{remark}

\newtheorem*{question}{Question}

\DeclareMathOperator{\nacc}{nacc}

\DeclareMathOperator{\pr}{Pr}

\DeclareMathOperator{\Int}{Int}

\DeclareMathOperator{\ran}{ran}

\DeclareMathOperator{\cf}{cf}

\newcommand{\sk}{\vskip.05in}
\DeclareMathOperator{\id}{id}

\newcommand{\restr}{\upharpoonright}

\newcommand{\subs}{\subseteq}

\DeclareMathOperator{\pp}{pp}

\numberwithin{equation}{section}
\begin{document}
\title{On idealized versions of $\pr_1(\mu^+,\mu^+,\mu^+,\cf(\mu))$}
\author{Todd Eisworth}
\address{Department of Mathematics\\
         Ohio University\\
         Athens, OH 45701}
\email{eisworth@math.ohiou.edu}
\date{\today}
\begin{abstract}
We obtain an improvement of some coloring theorems from \cite{nsbpr}, \cite{819}, and \cite{APAL} for the case where the singular cardinal in question has countable cofinality. As a corollary, we obtain an ``idealized'' version of the combinatorial principle $\pr_1(\mu^+,\mu^+,\mu^+,\cf(\mu))$ that maximizes the indecomposability of the associated ideal.
\end{abstract}

\maketitle
\section{Introduction}

Our aim in this paper is to provide the last piece of a proof of the following theorem:

\begin{theorem}
\label{mainthm}
Let $\mu$ be a singular cardinal. Then there are a coloring $c:[\mu^+]^2\rightarrow\mu^+$ and an ideal $I$ on $\mu^+$ such that
\begin{enumerate}
\item $I$ is a proper ideal containing all non-stationary subsets of~$\mu^+$ that is $\sigma$-indecomposable for all regular $\sigma<\mu$ with $\sigma\neq\cf(\mu)$, and
\sk
\item if $\langle t_\alpha:\alpha<\mu^+\rangle$ is a sequence of disjoint subsets of $\mu^+$ each of cardinality less than $\cf(\mu)$, then for $I$-almost all $\beta^*<\mu^+$ there are $\alpha<\beta<\mu^+$ such that $c$ is constant with value $\beta^*$ on $t_\alpha\times t_\beta$.
\sk
\end{enumerate}
\end{theorem}

If we remove all references to the ideal $I$ from the statement of Theorem~\ref{mainthm}, then what remains is a combinatorial principle known as $\pr_1(\mu^+,\mu^+,\mu^+,\cf(\mu))$.  This principle  states that there is a coloring $c:[\mu^+]^2\rightarrow\mu^+$ of the pairs drawn from $\mu^+$ with the property that whenever $\langle t_\alpha:\alpha<\mu^+\rangle$ is a disjoint collection of sets from $[\mu^+]^{<\cf(\mu)}$ and $\beta^*<\mu^+$, we can find $\alpha<\beta<\mu^+$ such that $c$ is constant with value $\beta^*$ when restricted to $t_\alpha\times t_\beta$.

Whether $\pr_1(\mu^+,\mu^+,\mu^+,\cf(\mu))$ necessarily holds for singular $\mu$ remains a mystery (although recent work of Rinot~\cite{assaf} has shown that it is equivalent to asking if the negative square-brackets relation $\mu^+\nrightarrow[\mu^+]^2_{\cf(\mu)}$ holds), but theorems like Theorem~\ref{mainthm} tell us that approximations to this principle are true, approximations that involve an ideal $I$ and colorings that always achieve almost every (with respect to the ideal $I$) value.  Furthermore, these ``idealized'' results have played an important role in analyzing the relationship between partition relations and reflection properties for successors of singular cardinals, with~\cite{impossible} being the latest example.

Our opening sentence mentioned that we are providing the ``last piece'' of a proof of Theorem~\ref{mainthm}.  This is because our work in \cite{nsbpr} proves Theorem~\ref{mainthm} for the case where $\mu$ has uncountable cofinality, and so the following result will finish the job:

\begin{theorem}
\label{thm1}
Let $\mu$ be a singular cardinal with $\cf(\mu)=\aleph_0$. Then there is a coloring $c:[\mu^+]^2\rightarrow\mu^+$ and an ideal $I$ on $\mu^+$ such that
\begin{enumerate}
\item $I$ is a proper ideal containing all non-stationary subsets of~$\mu^+$ that is $\sigma$-indecomposable for all regular uncountable $\sigma<\mu$, and
\sk
\item whenever $\langle t_\alpha:\alpha<\mu^+\rangle$ is a sequence of disjoint finite subsets of $\mu^+$, then for $I$-almost all $\beta^*<\mu^+$ there are $\alpha<\beta<\mu^+$ such that $c$ is constant with value $\beta^*$ on $t_\alpha\times t_\beta$.
\sk
\end{enumerate}
\end{theorem}

 Previous work has resulted in two ``near-misses'' to the above theorem.  In \cite{819}, Shelah and the author were able to obtain (1) together with a weak version on (2) in which each $t_\alpha$ is a singleton, while in \cite{APAL}, we were able to obtain (2) together with a slightly weaker version of (1) in which the associated ideal is $\sigma$-indecomposable for all regular $\sigma>\aleph_1$. Obtaining $\sigma$-indecomposability for all uncountable regular $\sigma$ and not just $\sigma>\aleph_1$ may seem a minimal gain, but we will argue at the end of the paper that the results captured by Theorem~\ref{mainthm} and Theorem~\ref{thm1} are in a certain sense the best possible results of this type.

 \section{Defining the ideal}

Our convention is that ideals on a cardinal $\kappa$ are proper and contain all bounded subsets of~$\kappa$.  We recall the following definition:
\begin{definition}
If $I$ is an ideal on a cardinal $\kappa$ and $\sigma$ is a regular cardinal, then we say $I$ is $\sigma$-indecomposable if $I$ is closed under {\em increasing} unions of length~$\sigma$.
\end{definition}

The ideal we use to prove Theorem~\ref{thm1} can be defined as soon as we have the appropriate club-guessing result at our disposal.  We start with the following theorem from~\cite{819}:

\begin{theorem}
\label{thm3}
Let $\lambda=\mu^+$ for $\mu$ a singular cardinal of countable cofinality, and let $S$ be a stationary subset of $\{\delta<\lambda:\cf(\delta)=\aleph_0\}$. Further suppose that
we have sequences $\langle c_\delta:\delta\in S\rangle$ and $\langle f_\delta:\delta\in S\rangle$ such that
\begin{enumerate}
\item $c_\delta$ is an increasing function from $\omega$ onto a cofinal subset of $\delta$ (for convenience, we define $c_\delta(-1)$ to be $-1$)
\item $f_\delta$ maps $\omega$ to the set of regular cardinals less than~$\mu$, and
\item for every closed unbounded $E\subs\lambda$ there are stationary many $\delta\in S$ such that $c_\delta(n)\in E$ for all $n<\omega$.
\end{enumerate}
Then there is an $S$-club system $\langle C_\delta:\delta\in S\rangle$ such that
\begin{enumerate}
\setcounter{enumi}{3}
\item $c_\delta(n)\in C_\delta$ for all $n$,
\item $|C_\delta\cap (c_\delta(n-1), c_\delta(n)]|\leq f_\delta(n)$, and
\item for every closed unbounded $E\subs\lambda$, there are stationarily many $\delta\in S$ such that
\begin{equation*}
(\forall n<\omega)(\exists\alpha\in\nacc(C_\delta)\cap E)[c_\delta(n-1)<\alpha<c_\delta(n)\text{ and }\cf(\alpha)>f_\delta(n)].
\end{equation*}
\end{enumerate}
\end{theorem}

Our ideal $I$ will be obtained from a specific instance of the preceding theorem, so let us fix the following objects:
\begin{itemize}
\item $\mu$ is a singular cardinal of countable cofinality,
\sk
\item $S$ is a stationary subset of $\mu^+$ consisting of ordinals of countable cofinality,
\sk
\item $\langle c_\delta:\delta\in S\rangle$ is a sequence satisfying assumptions (1) and (3) of the preceding theorem\footnote{Such a sequence exists by standard club-guessing results},
\sk
\item $\langle\mu_m:m<\omega\rangle$ is an increasing sequence of regular cardinals cofinal in~$\mu$, and
\sk
\item $\langle \cdot, \cdot\rangle:\omega\times\omega\rightarrow\omega$ is a bijection which is increasing in each component.
\end{itemize}

\begin{corollary}
\label{guessingsequence}
There is a sequence $\langle C_\delta:\delta\in S\rangle$ such that
\begin{enumerate}
\sk
\item $C_\delta$ is closed and unbounded in~$\delta$
\sk
\item $\ran(c_\delta)\subs C_\delta$
\sk
\item for each $m$ and $n$, $|C_\delta\cap (c_\delta(\langle m, n\rangle), c_\delta(\langle m, n\rangle+1])|\leq \mu_m$,
\sk
\item for each $m$ and $n$, $\nacc(C_\delta)\cap (c_\delta(\langle m, n\rangle), c_\delta(\langle m, n\rangle+1)]$ consists of ordinals of cofinality greater than $\mu_m$, and
\sk
\item for every closed unbounded $E\subs\mu^+$, there are stationarily many $\delta$ such that for every $m$ and $n$,
\begin{equation}
E\cap \nacc(C_\delta)\cap(c_\delta(\langle m, n\rangle), c_\delta(\langle m, n\rangle+1))\neq\emptyset.
\end{equation}
\sk
\end{enumerate}
\end{corollary}
\begin{proof}
For each $\delta\in S$, we define
\begin{equation}
f_\delta(\langle m, n\rangle+1)=\mu_m.
\end{equation}
The corollary now follows immediately from Theorem~\ref{thm3}. True, there is a minor issue in that requirement~(4) of the corollary seems to demand more than is given by Theorem~\ref{thm3}, but this is easily remedied by judiciously shrinking $C_\delta$.
\end{proof}

It is quite easy to picture the situation described in the preceding corollary.  Given $\delta\in S$, the function $c_\delta$ enumerates an $\omega$-sequence cofinal in $\delta$.  The set $\delta\setminus\ran(c_\delta)$ consists of countably many disjoint open intervals. Leaving out the first of these intervals, we see that each of the others is of the form $(c_\delta(\langle m, n\rangle), c_\delta(\langle m, n\rangle+1))$ for some natural numbers $m$ and $n$.  We will need some notation to help our discussion, so let us defne
\begin{equation}
\Int_\delta(m, n):=(c_\delta(\langle m, n\rangle), c_\delta(\langle m, n\rangle+1)).
\end{equation}

The intervals $\Int_\delta(m, n)$ are disjoint, and for fixed $m<\omega$, the intervals of the form $\Int_\delta(m, n)$ are unbounded in $\delta$.  Our construction guarantees
\begin{itemize}
\item $C_\delta\cap \Int_\delta(m, n)$ is of cardinality at most $\mu_m$,
\sk
\item $\nacc(C_\delta)\cap \Int_\delta(m, n)$ consists of ordinals of cofinality greater than $\mu_m$, and
\sk
\item for every club $E\subs \mu^+$, there are stationarily many $\delta\in S$ such that
\begin{equation}
(\forall m<\omega)(\forall n<\omega)[E\cap\nacc(C_\delta)\cap\Int_\delta(m, n)\neq\emptyset.]
\end{equation}
\end{itemize}

\begin{definition}
\label{idealdef}
Let $S$ and $\langle C_\delta:\delta\in S\rangle$ be as in Corollary~\ref{guessingsequence}.
\begin{enumerate}
\item Given $\delta\in S$, the ideal $I_\delta$ on $C_\delta$ is defined by
\begin{equation}
\label{ideltadefn}
A\in I_\delta\Longleftrightarrow (\forall^*m<\omega)(\forall^*n<\omega)[A\cap\nacc(C_\delta)\cap \Int_\delta(m, n)=\emptyset],
\end{equation}
where the notation ``$\forall^*m<\omega$'' means ``for all but finitely many $m<\omega$''.
\sk
\item The ideal $I$ on $\mu^+$ is defined by putting a set $A\subs\mu^+$ into $I$ if and only if  there is a closed unbounded $E\subs\mu^+$ such that
\begin{equation}
(\forall\delta\in E\cap S)[A\cap E\cap C_\delta\in I_\delta].
\end{equation}
\end{enumerate}
\end{definition}

We will prove an easy theorem about the above ideals, and afterwards make a few remarks setting the above definition in context.

\begin{theorem}
Suppose $\langle I_\delta:\delta\in S\rangle$ and $I$ are as in the preceding definition.
\begin{enumerate}
\item Each $I_\delta$ is an ideal on $C_\delta$ that is also $\sigma$-indecomposable for every uncountable regular cardinal~$\sigma$.
\sk
\item The ideal $I$ is a non-trivial ideal on $\mu^+$ extending the non-stationary ideal that is also $\sigma$-indecomposable for every uncountable regular $\sigma<\mu$.
\sk
\end{enumerate}
\end{theorem}
\begin{proof}
Proof of (1):  It should be clear that $I_\delta$ is an ideal on $C_\delta$, so we deal only with indecomposability.   Thus, let $\sigma$ be an uncountable regular cardinal, and suppose $\langle A_i:i<\sigma\rangle$ is a $\subseteq$-increasing sequence of sets from $I_\delta$. We show that the union of the sets $A_i$ is also in $I_\delta$.

Given $i<\sigma$, let us define $m_i$ to be the least natural number $m$ such that
\begin{equation}
(\forall j\geq m)(\forall^*n<\omega)[A_i\cap\nacc(C_\delta)\cap \Int_\delta(j, n)=\emptyset].
\end{equation}
It is clear that $m_i$ is defined by virtue of the definition of $I_\delta$.
and for $j\geq m$, define
\begin{equation}
f_i(j)=\max\{n<\omega: A_i\cap\nacc(C_\delta)\cap \Int_\delta(j, f(j))\neq\emptyset\}.
\end{equation}
Now set
\begin{equation}
B_i=C_\delta\cap [\cup\{\Int_\delta(j,k):j<m_i\text{ or }k\leq f_i(j)\}].
\end{equation}
It is clear that $B_i$ is a subset of $C_\delta$ in $I_\delta$, and we have ensured that $A_i\subs B_i$. It is also clear that the sequence $\langle B_i:i<\sigma\rangle$ is increasing, for
\begin{equation}
i_0<i_1\Longrightarrow m_{i_0}\leq m_{i_1}\text{ and }f_{i_0}(j)\leq f_{i_1}(j)\text{ whenever }m_{i_1}\leq j.
\end{equation}
Thus, it suffices to prove $\cup\{B_i:i<\sigma\}\in I_\delta$.

This is quite easy to do. By our observations above, we know the sequence $\langle m_i:i<\sigma\rangle$ is non-decreasing and hence must be eventually constant as $\sigma$ has uncountable cofinality. Therefore we can find $i^*<\sigma$ and $m^*<\omega$ such that
\begin{equation}
i^*\leq i<\sigma\Longrightarrow m_{i}=m^*.
\end{equation}
For each $j\geq m^*$, the sequence $\langle f_i(j):i^*\leq i<\sigma\rangle$ is also non-decreasing, hence eventually constant. Since $\sigma$ has uncountable cofinality, it follows that there is an ordinal $i^\dagger<\sigma$ greater than $i^*$ such that
\begin{equation}
f_i(j)=f_{i^\dagger}(j)\text{ whenever $i^\dagger\leq i<\sigma$ and $m^*\leq j$}.
\end{equation}
We immediately conclude
\begin{equation}
B_{i^\dagger}=B_i\text{ whenever }i^\dagger\leq i<\sigma,
\end{equation}
and therefore
\begin{equation}
\bigcup_{i<\sigma} B_i = B_{i^\dagger}\in I_\delta,
\end{equation}
as required.

\medskip

\noindent Proof of (2):  The ideal $I$ is non-trivial and extends the non-stationary ideal because of the club-guessing properties of $\langle C_\delta:\delta\in S\rangle$.  The indecomposability requirement is also easy (we could send the reader to Observation~3.2 on page~139 of~\cite{cardarith}) but we include a proof as this indecomposability of $I$ is one of the points of this paper.

Let $\sigma<\mu$ be an uncountable regular cardinal, and suppose $\langle A_\alpha:\alpha<\sigma\rangle$ is a $\subs$-increasing sequence of sets in $I$.   We will show that the set
\begin{equation}
A:=\bigcup_{\alpha<\sigma}A_\alpha
\end{equation}
is in $I$.

We do this by contradiction, so assume $A\notin I$.  This means that for every closed unbounded $E\subs\mu^+$, there is a $\delta\in E\cap S$ with
\begin{equation}
A\cap E\cap C_\delta\notin I_\delta.
\end{equation}

On the other hand, for each $\alpha<\sigma$ there is a club $E_\alpha\subseteq\mu^+$ such that
\begin{equation}
A_\alpha\cap E_\alpha\cap C_\delta\in I_\delta\text{ for all }\delta\in E\cap S.
\end{equation}
  Let
\begin{equation}
E:=\bigcap_{\alpha<\sigma}E_\alpha,
\end{equation}
and we have
\begin{equation}
(\forall\alpha<\sigma)(\forall\delta\in S\cap E)[A_\alpha\cap E\cap C_\delta\in I_\delta].
\end{equation}

Choose $\delta\in E\cap S$ with $A\cap E\cap C_\delta\notin I_\delta$.  Then
\begin{equation}
A\cap E\cap C_\delta = \bigcup_{\alpha<\sigma}A_\alpha\cap E\cap C_\delta,
\end{equation}
and we have contradicted the $\sigma$-indecomposability of $I_\delta$.

\end{proof}

Readers of previous work in this area may recognize the ideal $I$ as being of the form $\id_p(\bar{C},\bar{I})$ for the sequence $\bar{I}=\langle I_\delta:\delta\in S\rangle$. Ideals of this form were first introduced in~\cite{cardarith}, and they have played a fundamental role in the investigation of coloring theorems at successors of singular cardinals. For example, such ideals underly the proofs of previously established cases of Theorem~\ref{mainthm}.  One of the points of this paper is that we our sequence $\langle I_\delta:\delta\in S\rangle$ differs from what has been used before: in previous work, the ideal $\id_p(\bar{C},\bar{I})$ was constructed using an ideal known as  $J^{b[\mu]}_{C_\delta}$ in place of the ideal $I_\delta$ defined here. Replacing $J^{b[\mu]}_{C_\delta}$ by the ideal $I_\delta$ from Definition~\ref{idealdef} is the main new idea needed to obtain Theorem~\ref{thm1}.

\section{Defining the coloring}

We obtain the coloring using the techniques of~\cite{nsbpr} and~\cite{819}, combining scale combinatorics together with minimal walks.  In this section, we review a little notation from our prior work, and then define the coloring we use.

\bigskip

\noindent{\sf Scales}

\bigskip

\begin{definition}
Let $\mu$ be a singular cardinal. A {\em scale for
$\mu$} is a pair $(\vec{\mu},\vec{f})$ satisfying
\begin{enumerate}
\item $\vec{\mu}=\langle\mu_i:i<\cf(\mu)\rangle$ is an increasing sequence of regular cardinals
such that $\sup_{i<\cf(\mu)}\mu_i=\mu$ and $\cf(\mu)<\mu_0$.
\item $\vec{f}=\langle f_\alpha:\alpha<\mu^+\rangle$ is a sequence of functions such that
\begin{enumerate}
\item $f_\alpha\in\prod_{i<\cf(\mu)}\mu_i$.
\item If $\gamma<\delta<\beta$ then $f_\gamma<^* f_\beta$, where  the notation $f<^* g$  means that $\{i<\cf(\mu): g(i)\leq f(i)\}$ is bounded in $\cf(\mu)$.
\item If $f\in\prod_{i<\cf(\mu)}\mu_i$ then there is an $\alpha<\beta$ such that $f<^* f_\alpha$.
\end{enumerate}
\end{enumerate}
\end{definition}

We are going to need a couple of well-known functions associated with a given scale $(\vec{\mu},\vec{f})$.

\begin{definition}
Let $(\vec{\mu},\vec{f})$ be a scale for some singular cardinal $\mu$. We define functions $\Gamma$ and $\Gamma^+$ with domain included in $[\mu^+]^2$ as follows:
\begin{equation}
\Gamma(\alpha,\beta):=\sup\{i<\cf(\mu):f_\beta(i)\leq f_\alpha(i)\},
\end{equation}
and
\begin{equation}
\Gamma^+(\alpha,\beta):=
\begin{cases}
\max\{i<\cf(\mu):f_\beta(i)\leq f_\alpha(i)\}  &\text{if this maximum exists, and}\\
\text{undefined} &\text{otherwise.}
\end{cases}
\end{equation}
\end{definition}

Both $\Gamma$ and $\Gamma^+$ map their domains into $\cf(\mu)$, and they are equal whenever $\Gamma^+$ is defined.  The proof that our coloring works actually requires a great deal of scale combinatorics, but everything we need is encapsulated in the following lemma from~\cite{upgradesii}, which we quote without proof:

\begin{lemma}
\label{Gammalemma}
Assume $\mu$ is a singular cardinal and $(\vec{\mu},\vec{f})$ is a scale for $\mu$.
Further assume:
\begin{itemize}
\item $M_0\in M_1\in M_2$ are elementary submodels of $\mathfrak{A}$ of cardinality~$\mu$ such that $M_i\cap\mu^+$ is an initial segment of~$\mu^+$,
\sk
\item $(\vec{\mu},\vec{f})\in M_0$,
\sk
\item $\beta^*=M_0\cap\mu^+$,
\sk
\item $\bar{s}=\langle s_\alpha:\alpha<\mu^+\rangle$ is sequence of pairwise disjoint elements of $[\mu^+]^{<\cf(\mu)}$ with $\bar{s}\in M_0$, and
\sk
\item $t\in [\mu^+]^{<\cf(\mu)}$ with $M_2\cap\mu^+\leq\min(t)$.
\sk
\end{itemize}
Then for all sufficiently large $i<\cf(\mu)$, there are unboundedly many $\alpha<\beta^*$ such that for all $\epsilon_a\in s_\alpha$ and $\epsilon_b\in t$, we have
\begin{equation}
\Gamma^+(\epsilon_a,\epsilon_b)=i,
\end{equation}
but
\begin{equation}
f_{\beta^*}(i+1)<f_{\epsilon_a}(i+1).
\end{equation}
\end{lemma}

\bigskip

\noindent{\sf Minimal Walks}

\bigskip

Our coloring is also going to require some results from Todorcevic's theory of minimal walks. We start by recalling that $\bar{e}=\langle e_\alpha:\alpha<\lambda\rangle$ is a $C$-sequence for the cardinal
$\lambda$ if $e_\alpha$ is closed unbounded in $\alpha$ for each $\alpha<\lambda$.
Given $\alpha<\beta<\lambda$ the {\em minimal walk from $\beta$ to $\alpha$ along $\bar{e}$}
is defined to be the sequence $\beta=\beta_0>\dots>\beta_{n}=\alpha$ obtained by setting
\begin{equation}
\beta_{i+1}=\min(e_{\beta_i}\setminus\alpha).
\end{equation}
The function $\rho_2:[\lambda]^2\rightarrow\omega$ giving the length of the walk from $\beta$ to $\alpha$ will be quite important to us. More formally, we set
\begin{equation}
\rho_2(\alpha,\beta)=\text{ least $i$ for which $\beta_i(\alpha,\beta)=\alpha$}.
\end{equation}
Next, for $i\leq\rho_2(\alpha,\beta)$, we set
\begin{equation*}
\beta_i^-(\alpha,\beta)=
\begin{cases}
0 &\text{if $i=0$},\\
\sup(e_{\beta_j(\alpha,\beta)}\cap\alpha) &\text{if $i=j+1$ for $j<\rho_2(\alpha,\beta)$}.
\end{cases}
\end{equation*}
Clearly, for $0<i<\rho_2(\alpha,\beta)$, the ordinals $\beta^-_i(\alpha,\beta)$ and $\beta_i(\alpha,\beta)$ are consecutive elements in $e_{\beta_{i-1}(\alpha,\beta)}$, and together they delineate an interval which contains~$\alpha$.

Continuing our discussion, we define
\begin{gather}
\gamma(\alpha,\beta)=\beta_{\rho_2(\alpha,\beta)-1}(\alpha,\beta),\\
\intertext{and}
\gamma^-(\alpha,\beta)=\max\{\beta_i^-(\alpha,\beta):i<\rho_2(\alpha,\beta)\},
\end{gather}

Note that $\gamma^-(\alpha,\beta)<\alpha$, and if $\gamma^-(\alpha,\beta)<\alpha^*\leq\alpha$, then
\begin{equation}
\beta_i(\alpha,\beta)=\beta_i(\alpha^*,\beta)\text{ for $i<\rho_2(\alpha,\beta)$}.
\end{equation}

We do need to use a generalization of the minimal walks machinery in order to handle some issues that arise when dealing with successors of singular cardinals of countable cofinality.  These techniques were introduced by the author and Shelah in~\cite{819}, and they were further developed in~\cite{APAL}.

\begin{definition}
\label{generalizeddef}
Let $\lambda$ be a cardinal. A {\em generalized $C$-sequence} is a family
\begin{equation*}
\bar{e}=\langle e^m_\alpha:\alpha<\lambda, m<\omega\rangle
\end{equation*}
such that for each $\alpha<\lambda$ and $m<\omega$,
\begin{itemize}
\item $e^m_\alpha$ is closed unbounded in $\alpha$, and
\sk
\item $e^m_\alpha\subs e^{m+1}_\alpha$.
\sk
\end{itemize}
\end{definition}

One can think of a generalized $C$-sequence as a countable family of $C$-sequences which are increasing in a sense. One can also utilize generalized $C$-sequences in the context of minimal walks. In this paper, we do this in the simplest fashion --- given $m<\omega$ and $\alpha<\beta<\lambda$, we let the {\em $m$-walk from $\beta$ to $\alpha$ along $\bar{e}$} consist of the minimal walk from $\beta$ to $\alpha$ using the $C$-sequence $\langle e^m_\gamma:\gamma<\lambda\rangle$. Such walks have their associated parameters, and we use the superscript $m$ to indicate which part of the generalized $C$-sequence is being used in computations. So, for example, the $m$-walk from $\beta$ to $\alpha$ along $\bar{e}$ will have length $\rho_2^m(\alpha,\beta)$, and consist of ordinals denoted $\beta^m_i(\alpha,\beta)$ for $i\leq\rho^m_2(\alpha,\beta)$.

\bigskip

\noindent{\sf The coloring}

\bigskip

We are now in a position to define our coloring for a given singular cardinal~$\mu$.  The definition does not require that the cofinality of $\mu$ is countable, but it does need three parameters:

\begin{itemize}
\item a scale $(\vec{\mu},\vec{f})$,
\sk
\item  a generalized $C$-sequence $\bar{e}$, and
\sk
\item a bookkeeping function  $b:\cf(\mu)\rightarrow\omega$ with the property that for each $m<\omega$ there are arbitrarily large $i<\cf(\mu)$ with $b(i)=m$.
\sk
\end{itemize}

\begin{definition}
\label{coloringdefn}
For $\alpha<\beta<\mu^+$, we define
\begin{equation}
c(\alpha,\beta)=\beta^{m(\alpha,\beta)}_{k(\alpha,\beta)}(\alpha,\beta),
\end{equation}
where
\begin{equation}
m(\alpha,\beta):=b(\Gamma(\alpha,\beta)),
\end{equation}
and
\begin{equation}
\label{kdef}
k(\alpha,\beta):=\text{ least $k\leq\rho^{m(\alpha,\beta)}_2(\alpha,\beta)$ such that }\Gamma(\alpha,\beta^{m(\alpha,\beta)}_k)\neq\Gamma(\alpha,\beta).
\end{equation}
\end{definition}

The definition of $c$ is easier to understand in words:  Given $\alpha<\beta<\mu^+$, we use $\Gamma(\alpha,\beta)$ and $b$ to get a natural number $m=m(\alpha,\beta)$.  Next, we walk from $\beta$ down to $\alpha$ using the $C$-sequence $\langle e^m_\xi:\xi<\mu^+\rangle$, and we halt as soon as we hit a place where the value of $\Gamma$ changes. The ordinal where we stop is then the value assigned to $c(\alpha,\beta)$.  This is the same coloring used in \cite{nsbpr}, \cite{819}, and \cite{APAL}, but the proof that this coloring does what we want when $\mu$ has countable cofinality requires us to use the ideal $I$ constructed in the previous section.

\section{Proof of Theorem~\ref{thm1}}

Recall that our goal, Theorem~\ref{thm1} from the introduction, states the following:

\begin{thm}
Let $\mu$ be a singular cardinal with $\cf(\mu)=\aleph_0$. Then there is a coloring $c:[\mu^+]^2\rightarrow\mu^+$ and an ideal $I$ on $\mu^+$ such that
\begin{enumerate}
\item $I$ is a proper ideal containing all non-stationary subsets of~$\mu^+$ that is $\sigma$-indecomposable for all regular uncountable $\sigma<\mu$, and
\sk
\item whenever $\langle t_\alpha:\alpha<\mu^+\rangle$ is a sequence of disjoint finite subsets of $\mu^+$, then for $I$-almost all $\beta^*<\mu^+$ there are $\alpha<\beta<\mu^+$ such that $c$ is constant with value $\beta^*$ on $t_\alpha\times t_\beta$.
\sk
\end{enumerate}
\end{thm}
\begin{proof}

Our proof consists of putting together the pieces laid out in previous sections. We start by fixing a scale $(\vec{\mu},\vec{f})$ for $\mu$ and
a sequence $\langle c_\delta:\delta\in S\rangle$ satisfying conditions (1) and (3) of Theorem~\ref{thm3} with $S=\{\delta<\mu^+:\cf(\delta)=\aleph_0\}$.

Next, we take the sequence $\langle \mu_m:m<\omega\rangle$ of cardinals from our scale together with $\langle c_\delta:\delta\in S\rangle$ and apply Corollary~\ref{guessingsequence}. This gives us a club-guessing sequence $\bar{C}=\langle C_\delta:\delta\in S\rangle$. (The fact that the sequence $\langle \mu_m:m<\omega\rangle$ is used for both the scale and the club-guessing is just a convenience.)  Once we have $\bar{C}$, we define our ideal $I$ just as in the end of Section 2.

Our coloring will be as in  Definition~\ref{coloringdefn}.  This definition demands three parameters, one of which is our fixed scale $(\vec{\mu},\vec{f})$. Obtaining a suitable bookkeeping function $b:\omega\rightarrow\omega$ is no trouble at all, so we are left with deciding on a generalized $C$-sequence $\bar{e}$ to be used for our minimal walks. This turns out to be a critical point, as our $\bar{e}$ needs to be quite special in order for the proof to work.  The intent is that $\bar{e}$ should ``swallow'' the sequence $\langle C_\delta:\delta\in S\rangle$ in a certain sense, along the lines of what is achieved in Lemma~3.2 of~\cite{819}.

To make this precise, we need to recall a bit of notation from Section~2. Recall that given $\delta\in S$, we used the function $c_\delta$ to divide $\delta\setminus\ran(c_\delta)$ into intervals:
\begin{equation}
\Int_\delta(m, n)= (c_\delta(m, n), c_\delta(m, n)+1).
\end{equation}
Given $m<\omega$, let us define
\begin{equation}
C_\delta[m]=\ran(c_\delta)\cup\bigcup_{i\leq m}\bigcup_{n<\omega}(C_\delta\cap \Int_\delta(i,n)),
\end{equation}
that is, $C_\delta[m]$ consists of the cofinal $\omega$-sequence $\ran(c_\delta)$ together with those parts of $C_\delta$ that lie in intervals of the form $\Int_\delta(i,n)$ for some $i\leq m$ and $n<\omega$. We note that $C_\delta[m]$ is club in $\delta$ of cardinality at most $\mu_m$, and the sequence $\langle C_\delta[m]:m<\omega\rangle$ is $\subseteq$-increasing with union $C_\delta$.

Given $m<\omega$ and $\alpha<\mu^+$, we will obtain $e^m_\alpha$ as the closure in $\alpha$ of a union of approximations $e^m_\alpha[\beta]$ for $\beta<\omega_1$.  Start out by letting $e_\alpha$ be club in $\alpha$ of order-type $\cf(\alpha)$.  We then set

\begin{eqnarray*}
e^0_\alpha[0]&= &e_\alpha\\
e^m_\alpha[\beta+1] &= &\text{ closure in $\alpha$ of }e^m_\alpha[\beta]\cup \bigcup_{\delta\in S\cap e^m_\alpha[\beta]}C_\delta[m]\\
e^{m+1}_\alpha[0] & = &e^m_\alpha\\
e^m_\alpha[\beta] &= &\text{ closure in $\alpha$ of
}\bigcup_{\gamma<\beta}e^m_\alpha[\gamma]\text{ for $\beta$ limit}\\
e^m_\alpha &= &\text{ closure in $\alpha$ of }\bigcup_{\beta<\omega_1}e^m_\alpha[\beta].
\end{eqnarray*}

The following easy lemma captures the salient properties of the above construction:

\begin{lemma}
\label{ladderlemma}
The collection $\bar{e}=\langle e_\alpha^m:\alpha<\mu^+, m<\omega\rangle$ is a generalized $C$-sequence with the following properties:
\begin{enumerate}
\item $|e^m_\alpha|\leq \aleph_1+\cf(\alpha)+\mu_m$,
\sk
\item if $\delta\in e^m_\alpha\cap S$, then $C_\delta[m]\subseteq e^m_\alpha$, and
\sk
\item for all $\alpha<\mu^+$, for all sufficiently large $m<\omega$, if $\delta\in e^m_\alpha\cap S$, then for all $n<\omega$
\begin{equation}
\nacc(C_\delta)\cap \Int_\delta(m, n)\subseteq\nacc(e^m_\alpha).
\end{equation}
\end{enumerate}
\end{lemma}
\begin{proof}
Our construction guarantees that $e^m_\alpha\subseteq e^{m+1}_\alpha$, so $\bar{e}$ is indeed a generalized $C$-sequence.  The first statement follows easily from the construction, as $C_\delta[m]$ is of cardinality at most $\mu_m$ by our choice of $\bar{C}$.  Statement (2) is also guaranteed by our construction:  if $\delta\in S\cap e^m_\alpha$ then $\delta\in S\cap e^m_\alpha[\beta]$ for some $\beta<\omega_1$ as the cofinality of $\delta$ is countable.

Given $\alpha<\mu^+$, suppose $m$ is such that $\cf(\alpha)+\aleph_1\leq\mu_m$ (this happens for all sufficiently large $m$ given our choice of $\vec{\mu}$).  We know from (1) that $|e^m_\alpha|\leq\mu_m$, and if $\delta\in e^m_\alpha$ then (2) guarantees that $C_\delta[m]\subseteq e^m_\alpha$.

Suppose now that $\beta\in\nacc(C_\delta)\cap\Int_\delta(m, n)$ for some $n<\omega$.  Then $\beta\in C_\delta[m]$ by definition, and hence in $e^m_\alpha$ as well.  The cofinality of $\beta$ is greater than $\mu_m$ (again, by choice of $\bar{C}$) and therefore $\beta$ cannot be a point of accumulation of $e^m_\alpha$. The result follows immediately.
\end{proof}

Where are we now?  The coloring we need is as defined at the end of the previous section using our scale $(\vec{\mu},\vec{f})$ and our generalized $C$-sequence $\bar{e}$ (and $b:\omega\rightarrow\omega$) as parameters. The ideal $I$ is defined as in Section 2 using $\bar{C}$ and $\vec{\mu}$ as parameters.  We must now check that these objects have the required properties.

Let $\langle t_\alpha:\alpha<\mu^+\rangle$ be a pairwise disjoint collection of finite subsets of~$\mu^+$. By passing to a subsequence if necessary, we may assume without loss of generality that
\begin{equation}
\alpha<\beta<\mu^+\Longrightarrow \max(t_\alpha)<\min(t_\beta).
 \end{equation}
Next us agree to call an ordinal $\beta^*<\mu^+$  {\em attainable} if we can find $\alpha<\beta$ for which
\begin{equation}
c\restr t_\alpha\times t_\beta\text{ is constant with value }\beta^*.
\end{equation}
We must show that $I$-almost all ordinals $\beta^*<\mu^+$ are attainable.  To do this, we must produce a closed unbounded $E\subseteq\mu^+$ with the property that whenever $\delta\in S$ satisfies $E\cap C_\delta\notin I_\delta$, then $I_\delta$-almost all members of $E\cap C_\delta$ are attainable.
Obtaining the club $E$ requires us to consider elementary submodels, so let $\chi$ be a sufficiently large regular cardinal, and let $\langle M_\xi:\xi<\mu^+\rangle$ be a continuous $\in$-increasing chain of elementary submodels of $H(\chi)$ such that
\begin{itemize}
\item $\langle t_\alpha:\alpha<\mu^+\rangle$ together with all parameters needed to define $I$ and $c$ are in~$M_0$,
\sk
\item $\langle M_\zeta:\zeta\leq\xi\rangle\in M_{\xi+1}$,
\sk
\item $M_\xi\cap\mu^+$ is an initial segment of $\mu^+$.
\sk
\end{itemize}
We now define our club $E\subseteq\mu^+$ by
\begin{equation}
E:=\{\delta<\mu^+:\delta= M_\delta\cap\mu^+\}.
\end{equation}

We must show that if $\delta\in S$ and  $E\cap C_\delta\notin I_\delta$, then $I_\delta$-almost every $\beta^*\in E\cap C_\delta$ is attainable.
Fix such a $\delta$ and choose $\beta<\mu^+$ with $\delta<\min(t_\beta)$, say
\begin{equation}
t_\beta=\{\epsilon_j:j<j^*\}.
\end{equation}

Since $t_\beta$ is finite and $\bar{e}$ is a generalized $C$-sequence (in particular, since $e^m_\alpha\subseteq e^{m+1}_\alpha$ for all $\alpha<\mu^+$ and $m<\omega$), we can choose $m_0<\omega$ large enough so that for all $m\geq m_0$ and $j<j^*$,
\begin{equation}
\rho^m_2(\delta,\epsilon_j)=\rho^{m_0}_2(\delta,\epsilon_j),
\end{equation}
and
\begin{equation}
\beta^m_i(\delta,\epsilon_j)=\beta^{m_0}_i(\delta,\epsilon_j)\text{ for all }i<\rho^m_2(\delta,\epsilon_j).
\end{equation}
(In the vocabulary of \cite{819}, we say that such an $m_0$ ``settles all the walks from $t_\beta$ down to $\delta$''.)

Next, we define
\begin{equation}
\gamma_j:=\beta^{m_0}_{\rho_2^{m_0}(\delta,\epsilon_j)-1}(\delta,\epsilon_j)\text{ for }j<j^*.
\end{equation}

In summary, given $j<j^*$, we know that for any $m\geq m_0$ the $m$-walk from $\epsilon_j$ down to $\delta$ is exactly the same as the $m_0$-walk from $\epsilon_j$ down to $\delta$, and $\gamma_j$ is the penultimate step of this walk.

Note that this means that $\delta\in e^m_{\gamma_j}$ for all $m\geq m_0$ and $j<j^*$, and so by (3) of Lemma~\ref{ladderlemma} we can find $m^*\geq m$ so large that
\begin{equation}
(\forall m\geq m^*)(\forall j<j^*)(\forall n<\omega)[\nacc(C_\delta)\cap\Int_\delta(m, n)\subseteq\nacc(\gamma_j)].
\end{equation}
Our intent is to prove that for any $m\geq m^*$, for all but finitely many $n<\omega$, any ordinal $\beta^*$ in $\nacc(C_\delta)\cap E\cap\Int_\delta(m, n)$ is attainable (see Proposition~\ref{attain} below).

Given $m\geq m^*$, we define
\begin{equation}
\gamma^-_m:=\max\{\beta^{m,-}_i(\delta,\epsilon_j):j<j^*\wedge i<\rho^m_2(\delta,\epsilon_j)\}.
\end{equation}
We note that $\gamma^-_m<\delta$, and if $\gamma^-_m<\epsilon<\delta$ then
\begin{equation}
\label{walkstuff}
(\forall j<j^*)(\forall i<\rho^m_2(\delta,\epsilon_j)[\beta^m_i(\epsilon,\epsilon_j)=\beta^m_i(\delta,\epsilon_j)].
\end{equation}

Since $\gamma_m^-<\delta$, we can define
\begin{equation}
n(m):=\min\{n<\omega:\gamma_m^-<c_\delta(\langle m, n\rangle)\}.
\end{equation}
This leads us now to the heart of the matter:
\bigskip

\begin{proposition}
\label{attain}
If $\beta^*\in E\cap \nacc(C_\delta)\cap \Int_\delta(m, n)$ for some $m\geq m^*$ and $n\geq n(m)$, then $\beta^*$ is attainable.
\end{proposition}
\begin{proof}

Fix $m\geq m^*$ and $n\geq n(m)$ and suppose $\beta^*\in\nacc(C_\delta)\cap \Int_\delta(m, n)$. Given our choice of $m^*$, we know (see the proof of (3) in Lemma~\ref{ladderlemma})
\begin{equation}
(\forall j<j^*)[\beta^*\in\nacc(e^m_{\gamma_j})],
\end{equation}
and so if we define
\begin{equation}
\eta_m^-:=\max\{\sup (e^m_{\gamma_j}\cap\beta^*):j<j^*\},
\end{equation}
we obtain
\begin{equation}
\gamma^-_m< c_\delta(\langle m, m(n)\rangle)\leq c_\delta(\langle m, n\rangle)\leq \eta_m^-<\beta^*< c_\delta(\langle m, n\rangle+1).
\end{equation}
(We get $c_\delta(\langle m, n\rangle)\leq \eta_m^-$ as the range of $c_\delta$ is a subset of each $e^m_{\gamma_j}$.)

Next, we claim that  $\eta_m^-<\epsilon<\beta^*$, then for any $j<j^*$ we have
\begin{equation}
\label{morewalk}
\beta^m_{\rho^m_2(\delta,\epsilon_j)}(\epsilon,\epsilon_j)=\beta^*.
\end{equation}

To see this, suppose $\eta_m^-<\epsilon<\beta^*$.  Then certainly we have $\gamma_m^-<\epsilon<\beta^*<\delta$ and so
\begin{equation}
(\forall j<j^*)(\forall i<\rho^m_2(\delta, \epsilon_j)[\beta^m_i(\epsilon,\epsilon_j)=\beta^m_i(\delta, \epsilon_j)=\beta^m_i(\beta^*,\epsilon_j)]
\end{equation}
by (\ref{walkstuff}).  In particular,
\begin{equation}
\beta^m_{\rho^m_2(\delta,\epsilon_j)-1}(\epsilon,\epsilon_j)=\gamma_j.
\end{equation}
But
\begin{equation}
\sup(e^m_{\gamma_j}\cap\beta^*)\leq\eta^-_m<\epsilon <\beta^*,
\end{equation}
 and therefore
\begin{equation}
\beta^* = \min(e^m_{\gamma_j}\setminus\epsilon),
\end{equation}
which establishes~(\ref{morewalk}).

In summary, the $\beta^*$ we are working with has the property that for all sufficiently large $\epsilon<\beta^*$, for all $j<j^*$, the $m$ walk from
$\epsilon_j$ down to $\epsilon$ passes through $\beta^*$.

Our proof now shifts from properties of minimal walks to the consideration of scale combinatorics.  We apply Lemma~\ref{Gammalemma} to the objects
\begin{itemize}
\item $M_{\beta^*}\in M_{\beta^*+1}\in M_\delta$,
\sk
\item $(\vec{\mu},\vec{f})$,
\sk
\item $\beta^*= M_{\beta^*}\cap \mu^+$ (recall that $\beta^*$ is in $E$)
\sk
\item $\bar{s}=\langle t_\alpha:\alpha<\mu^+\rangle$, and
\sk
\item $t:=\{\beta^m_i(\delta,\epsilon_j):i <\rho^m_2(\delta,\epsilon_j)\wedge j<j^*\}$,
\sk
\end{itemize}
and this allows us to choose $i^*<\omega$ and $\alpha<\beta^*$ such that
\begin{itemize}
\item $b(i^*)=m$
\sk
\item $\eta^-_m<\min(t_\alpha)\leq\max(t_\alpha)<\beta^*$
\sk
\item for all $\epsilon\in t_\alpha$, $j<j^*$, and $i<\rho^m_2(\delta,\epsilon_j)$, we have
\begin{equation}
\Gamma^+(\epsilon, \beta^m_i(\delta,\epsilon_j))= i^*,
\end{equation}
and
\item $f_{\beta^*}(i^*+1)<f_\epsilon(i^*+1)$.
\sk
\end{itemize}
(The second clause of the above deserves a little comment:  since $\beta^*=M_{\beta^*}\cap\mu^+$ and $t_\alpha$ is finite, we know that $\max(t_\alpha)<\beta^*$ whenever $\alpha<\beta^*$.)

To finish the proof of Proposition~\ref{attain}, we need to establish the following:
\begin{equation}
(\forall\epsilon\in t_\alpha)(\forall j<j^*)[c(\epsilon,\epsilon_j)=\beta^*].
\end{equation}

So suppose $\epsilon\in t_\alpha$ and $j<j^*$.  We have arranged things so that
\begin{equation}
m(\epsilon,\epsilon_j)=m,
\end{equation}
and therefore $c(\epsilon,\epsilon_j)$ will be computed using $m$-walks for our fixed value of $m$.

The definition of $c$ tells us to $m$-walk from $\epsilon_j$ down to $\epsilon$, computing $\Gamma$ at each step, and then stopping as soon as the value of $\Gamma$ changes. Since $\gamma^-_m\leq\eta^-_m<\epsilon<\delta$, we know by preceding work that
\begin{equation}
(\forall i<\rho^m_2(\delta, \epsilon_j))[\beta^m_i(\epsilon,\epsilon_j)=\beta^m_i(\delta,\epsilon_j)],
\end{equation}
and
\begin{equation}
\label{betastar}
\beta^m_{\rho^m_2(\delta,\epsilon_j)}(\epsilon,\epsilon_j)=\beta^*.
\end{equation}

Each of the ordinals $\beta^m_i(\delta,\epsilon_j)$ lies in the set $t$ to which we applied Lemma~\ref{Gammalemma}, and so we know
\begin{equation}
(\forall i<\rho^m_2(\delta,\epsilon_j))[\Gamma(\epsilon,\beta^m_i(\epsilon,\epsilon_j))=i^*].
\end{equation}

On the other hand, we know that (\ref{betastar}) holds, and Lemma~\ref{Gammalemma} made sure that we have
\begin{equation}
f_{\beta^*}(i^*+1)<f_\epsilon(i^*+1).
\end{equation}
Looking back at~(\ref{kdef}), we see
\begin{equation}
k(\epsilon,\epsilon_j)=\rho^m_2(\delta,\epsilon_j),
\end{equation}
and therefore
\begin{equation}
c(\epsilon,\epsilon_j)=\beta^{m(\epsilon,\epsilon_j)}_{k(\epsilon,\epsilon_j)}(\epsilon,\epsilon_j)=\beta^m_{\rho^m_2(\delta,\epsilon_j)}(\epsilon,\epsilon_j)=\beta^*,
\end{equation}
which finishes the proof of both Proposition~\ref{attain} and Theorem~\ref{thm1}.
\end{proof}
\end{proof}

\section{Conclusion}

We mentioned in the introduction that the conclusion of  Theorem~\ref{mainthm} is in some sense optimal, and we use this concluding section to discuss this a bit.

First, it is entirely possible that {\sf ZFC} proves $\pr_1(\mu^+,\mu^+,\mu^+,\cf(\mu))$ for every singular cardinal. We conjecture that this is not so, but if it turns out to be the case then clearly theorems like those considered here will be superceded.  On the other hand, if we have a singular cardinal~$\mu$ for which $\pr_1(\mu^+,\mu^+,\mu^+,\cf(\mu))$ fails, then Theorem~\ref{mainthm} is optimal in the sense that the ideal $I$ cannot be ``more indecomposable'':

\begin{proposition}
Let  $\mu$ be a singular cardinal, and suppose there are a coloring $c:[\mu^+]^2\rightarrow\mu$ and an ideal $I$ on $\mu^+$ such that
\begin{enumerate}
\sk
\item $I$ is a proper $\cf(\mu)$-indecomposable ideal on $\mu^+$ (containing the bounded subsets of $\mu^+$), and
\sk
\item if $\langle t_\alpha:\alpha<\mu^+\rangle$ is a sequence of disjoint subsets of~$\mu^+$ each of cardinality less than $\cf(\mu)$, then for $I$-almost all $\beta^*<\mu^+$ there are $\alpha<\beta<\mu^+$ such that $c$ is constant with value $\beta^*$ on $t_\alpha\times t_\beta$.
\sk
\end{enumerate}
Then $\pr_1(\mu^+,\mu^+,\mu^+,\cf(\mu))$ holds.
\end{proposition}
\begin{proof}
Suppose not.  Then $I$ is weakly $\mu^+$-saturated, that is, we cannot partition $\mu^+$ into $\theta$ disjoint $I$-positive sets.  (This implication is easy:  if $p:\mu^+\rightarrow\mu^+$ is a partition of $\mu^+$ into $\mu^+$ $I$-positive sets, then the composition $p\circ c$ shows us that $\pr_1(\mu^+,\mu^+,\mu^+,\cf(\mu))$ holds.) By Proposition~4.1 of~\cite{impossible}, the existence of a $\cf(\mu)$-indecomposable weakly $\mu^+$-saturated proper ideal on $I$ (containing all bounded subsets of $\mu^+$) implies $\pp(\mu)=\mu^+$, and once we have $\pp(\mu)=\mu^+$ we know $\pr_1(\mu^+,\mu^+,\mu^+,\cf(\mu))$ follows by Corollary~6.2 of~\cite{upgradesii}.
\end{proof}

Ideals on $\mu^+$ that extend the bounded ideal are automatically $\mu^+$-decomposable (as witnessed by the initial segments of $\mu^+$),  and $\sigma$-indecomposable for any regular $\sigma>\mu^+$ (as any increasing union of subsets of $\mu^+$ of length $\sigma$ must be eventually constant). Thus, the indecomposability obtained for the ideal $I$ in Theorem~\ref{mainthm} is optimal, unless the full version of $\pr_1(\mu^+,\mu^+,\mu^+,\cf(\mu))$ already holds, in which case Theorem~\ref{mainthm} is uninteresting.

We close with a related open question concerning idealized colorings:

\begin{question}
Our work in~\cite{upgradesi} and~\cite{upgradesii} combined with a coloring discovered by Assaf Rinot (see~\cite{assaf}) shows that the principle $\pr_1(\mu^+,\mu^+,\mu^+,\cf(\mu))$ is equivalent to the negative square brackets partition relation $\mu^+\nrightarrow[\mu^+]^2_{\mu^+}$.  Does this equivalence continue to hold for ``idealized'' colorings? That is, suppose $\mu$ is singular and we have an ideal $I$ and coloring $c:[\mu^+]^2\rightarrow\mu^+$ with the property that for any unbounded $A\subseteq \mu^+$, for $I$-almost all $\beta^*<\mu^+$ there are $\alpha<\beta$ in $A$ with $c(\alpha,\beta)=\beta^*$.  Is it the case that there is a coloring $d$ of the pairs from $\mu^+$ with the property that whenever $\langle t_\alpha:\alpha<\mu^+\rangle$ is a family of pairwise disjoint subsets of $\mu^+$ each of cardinality less than $\cf(\mu)$, then for $I$-almost all $\beta^*<\mu^+$ there are $\alpha<\beta$ such that the restriction of $d$ to $t_\alpha\times t_\beta$ is constant with value $\beta^*$?
\end{question}

\bibliographystyle{plain}

\end{document}